\documentclass[a4paper,11pt]{article}
\usepackage[utf8]{inputenc}
\usepackage[T1]{fontenc}

\usepackage{amsthm,amsmath}
\usepackage{tikz}
\usepackage{mathrsfs,amssymb,amsfonts} 
\usepackage{enumitem}
\usepackage{fullpage}
\usepackage{hyperref, enumerate}
\usepackage{graphicx}
\usepackage[babel]{microtype}
\usepackage[english]{babel}
\usepackage[capitalise]{cleveref}

\usepackage{thmtools}
\usepackage{mathtools, comment}
\usepackage{amssymb}
\usepackage[nomath]{lmodern}
\usepackage{graphicx}
\usepackage{pgf,tikz,tkz-graph,subcaption}
\usetikzlibrary{arrows,shapes}
\usetikzlibrary{decorations.pathreplacing}
\usepackage{tkz-berge}
\usepackage{enumitem}
\usepackage[normalem]{ulem}
\usepackage{hyperref}
\usepackage{algorithm}
\usepackage{algorithmic}
\hypersetup{colorlinks = true, linkcolor = blue, citecolor = blue, urlcolor = blue}
\newcommand*{\ceilfrac}[2]{\mathopen{}\left\lceil\frac{#1}{#2}\right\rceil\mathclose{}}
\newcommand*{\floorfrac}[2]{\mathopen{}\left\lfloor\frac{#1}{#2}\right\rfloor\mathclose{}}

\newcommand*{\abs}[1]{\lvert #1\rvert}

\allowdisplaybreaks

\usepackage[margin=1in]{geometry}
\newtheorem{defi}{Definition}
\newtheorem{conj}[defi]{Conjecture}

\newtheorem{thr}[defi]{Theorem}

\newtheorem{prop}[defi]{Proposition}

\newtheorem{claim}[defi]{Claim}
\newcommand*{\myproofname}{Proof}
\newenvironment{claimproof}[1][\myproofname]{\begin{proof}[#1]}{\end{proof}}

\allowdisplaybreaks

\usepackage[margin=1in]{geometry}
\newcommand{\diam}{diam}

\title{Sharp results for the Erd{\H{o}}s, Pach, Pollack and Tuza problem}

\date{}
\author{
Stijn Cambie \thanks{Department of Computer Science, KU Leuven Campus Kulak-Kortrijk, 8500 Kortrijk, Belgium. Supported by FWO grants with grant numbers 1225224N and 1222524N. E-mail: \protect\href{mailto:stijn.cambie@hotmail.com}{\protect\nolinkurl{stijn.cambie@hotmail.com}} and \protect\href{jorik.jooken@kuleuven.be}{\protect\nolinkurl{jorik.jooken@kuleuven.be}}}
\and
Jorik Jooken \footnotemark[1] 
}

\begin{document}
\maketitle
\begin{abstract}
    We consider the Erd{\H{o}}s, Pach, Pollack and Tuza problem, asking for the maximum diameter of a graph with given order $n$, minimum degree $\delta$ and clique number at most $\omega$. 
    We solve their problem asymptotically for the first hard case, $\omega \leq 3$, for the smallest values of $\delta$ by determining the smallest rational number $f(\delta)$ such that $\diam(G) \leq f(\delta)n+O(1)$ for all graphs $G$ with order $n$, minimum degree $\delta$ and clique number $\omega \leq 3$. We also consider the weaker version where the clique number $\omega \leq 3$ is replaced by having chromatic number $\chi \leq 3$ and solve this version for small $\delta$, thereby yielding a counterexample to a conjecture of Erd{\H{o}}s et al. in a regime where this conjecture was still open. When restricting the conjecture to graphs with chromatic number $\chi \leq 3$, we show that this counterexample appears for the smallest possible $\delta$, namely $\delta=16.$
\end{abstract}

\section{Introduction}
Solving a question by Gallai, in~\cite{EPPT89} Erd{\H{o}}s, Pach, Pollack and Tuza determined the maximum diameter of a graph given its minimum degree and order asymptotically. Furthermore, they also did this for the class of triangle-free and $C_4$-free graphs. 
% In each case, the maximum radius under the same constraints differs by a multiplicative factor $2$ (and possibly a small additive term).
We summarise~\cite[Thm.~1\&2]{EPPT89}

\begin{thr}\label{thr:EPPT_main}(\cite{EPPT89})
    A connected graph $G$ with minimum degree $\delta$ and order $n$ satisfies $\diam(G)\le 3 \frac{n}{\delta+1}+O(1).$
    If additionally $\omega(G)\le 2,$ then $\diam(G)\le 2\frac{n}{\delta}+O(1).$
\end{thr}
The ratio $\diam(G)/n$ cannot be further improved: sharpness can be derived from blowing up the vertices of a path by complete graphs or complete bipartite graphs.
We present a more elegant proof for the triangle-free case, avoiding case distinctions, for clarity.

\begin{proof}[Proof of~\cref{thr:EPPT_main} for $\omega \leq 2$]
    Let $v_0, v_d \in V(G)$ be such that $d=\diam(G)=d(v_0,v_d)$ and $N_i:=N_i(v_0)=\{x \in V(G) \mid d(v_0,x)=i\}.$
    For every $0 \le i \le d-3,$ $\abs{N_i}+\abs{N_{i+1}}+\abs{N_{i+2}}+\abs{N_{i+3}}\ge 2\delta,$ since the neighbourhoods of the endvertices $u$ and $v$ of an edge $uv \in N_{i+1} \times N_{i+2}$ have to be disjoint.
    Summing all these inequalities leads to $4n > \sum_{i=0}^{d-3} \left(\abs{N_i}+\abs{N_{i+1}}+\abs{N_{i+2}}+\abs{N_{i+3}}\right)\ge 2\delta ( d-2).$
\end{proof}

Since the sharp graphs for~\cref{thr:EPPT_main} have large clique number, they also considered similar statements with restrictions on clique number, i.e., bounding $\omega(G).$ 
They formulated the following conjecture:

\begin{conj}[\cite{EPPT89}]
\label{conj:EPPT_wrong}
Let $r, \delta \geq 2$ be fixed integers and let $G$ be a connected graph of order $n$ and minimum degree $\delta$.

\begin{itemize}
    \item[(i)] If $G$ is $K_{2r}$-free and $\delta$ is a multiple of $(r - 1)(3r + 2)$, then, as $n \to \infty$,
    \[
    \diam(G) \leq \frac{2(r - 1)(3r + 2)}{(2r^2 - 1)} \cdot \frac{n}{\delta} + O(1)
    \]
    \[
    = \left( 3 - \frac{2}{2r - 1} - \frac{1}{(2r - 1)(2r^2 - 1)} \right) \frac{n}{\delta} + O(1).
    \]
    
    \item[(ii)] If $G$ is $K_{2r+1}$-free and $\delta$ is a multiple of $3r - 1$, then, as $n \to \infty$,
    \[
    \text{diam}(G) \leq \frac{3r - 1}{r} \cdot \frac{n}{\delta} + O(1)
    \]
    \[
    = \left( 3 - \frac{2}{2r} \right) \frac{n}{\delta} + O(1).
    \]
\end{itemize}
\end{conj}

In~\cite{CSS21}, Czabarka, Singgih, and Sz\'ekely gave counterexamples to~\cref{conj:EPPT_wrong} (i) for every $r \geq 2$ and $\delta > 2(r-1)(3r+2)(2r-3)$ (leaving the regime $(r-1)(3r+2) \leq \delta \leq 2(r-1)(3r+2)(2r-3)$ still open). They subsequently stated an updated version of the conjecture, which no longer requires cases.

\begin{conj}[\cite{CSS21}]
    For every $k \ge 3$ and $\delta \ge \ceilfrac{3k}{2}-1$, if G is a connected graph of order $n$, minimum degree at least $\delta$ and $\omega(G)\le k$ (weaker version $\chi(G)\le k)$, then $\diam(G) \le \left(3-\frac 2k\right) \frac n{\delta}+O(1).$
\end{conj}

% A weaker bound of the weaker version has been proven by the same authors in~\cite{CSS21ejc}.
The weaker version has been proven for $k \in\{3,4\}$ in~\cite{CDS09} (for $k=4$) and~\cite{CSS23}. In this weaker version the colour classes of the $i^{th}$ neighbourhoods give information on the structure of the graphs with most edges (maximising the minimum degree), implying one can conclude more compared to the version where the clique number is bounded. 
% after which the authors used linear programming.

Let $G$ be a graph and $N_0 \subset V(G)$ be a subset of vertices. For all $i \geq 1$, define $N_i(N_0) := \{v \in V(G)~|~\min_{u \in N_0}(d(u,v))=i\}$ (we will refer to these vertex sets as \textit{layers} and omit the argument $N_0$ if it is clear from the context or not important how $N_0$ is chosen). If $N_0 = \{u\}$, we simply write $N_i(u)$ instead of $N_i(\{u\})$. Let $d$ be the largest integer such that $N_d$ is not empty. For $X \subset V$, the graph $G[X]=\left(X, E \cap \binom{X}{2}\right)$ is the subgraph of $G=(V,E)$ induced by $X.$ For a graph $G$ with $d \geq 2$, where $G[N_0 \cup N_1] \cong G[N_{d-1} \cup N_d]$ and the isomorphism maps vertices in $N_0$ to $N_{d-1}$ and vertices in $N_1$ to $N_d$, we define the \textit{concatenation} $G'$ as the graph obtained by taking the disjoint union of $G$ and $G-N_0-N_1$, indexing consecutive layers of $G'$ that originate from $G$ as $N'_i=N_i$ ($0 \leq i \leq d$) and those that originate from $G-N_0-N_1$ as $N'_{d-1+i}=N_i$ ($2 \leq i \leq d$), and adding edges between the last layer of $G$ and the first layer of $G-N_0-N_1$ such that $G'[N'_0 \cup N'_1 \cup N'_2] \cong G'[N'_{d-1} \cup N'_{d} \cup N'_{d+1}]$, where for each $k \in \{0,1,2\}$ the isomorphism should map vertices in $N'_k$ to vertices in $N'_{d-1+k}$. An example of a graph $G$ and its concatenation $G'$ is given in~\cref{fig:GAndConcatenation}. For integers $\delta$ and $\omega$ (or $\chi$), we call $G$ induced by consecutive layers $N_0, \ldots, N_d$ \textit{repeatable} (with respect to $\delta,\omega$ or $\delta,\chi$) if $d \geq 2$, $G[N_0 \cup N_1] \cong G[N_{d-1} \cup N_d]$ such that the isomorphism maps vertices in $N_0$ to $N_{d-1}$ and vertices in $N_1$ to $N_d$ and $G$ has clique number at most $\omega$ (or chromatic number at most $\chi$) and every vertex in $G$, except for vertices in its first and last layer, has degree at least $\delta$. Remark that the concatenation $G'$ of a repeatable graph $G$ is again repeatable. We call the integer $d-1$ the \textit{repetition length} of $G$. We call a graph a \textit{fundamental block} if it is obtained by deleting the vertices in the first and the last layer of a repeatable graph. For example, the graph $G$ shown in~\cref{fig:GAndConcatenation} is repeatable with respect to $\delta=3$ and $\omega=3$ (or $\chi=3$), has repetition length $3$ and the graph induced by $N_1 \cup N_2 \cup N_3$ is a fundamental block.

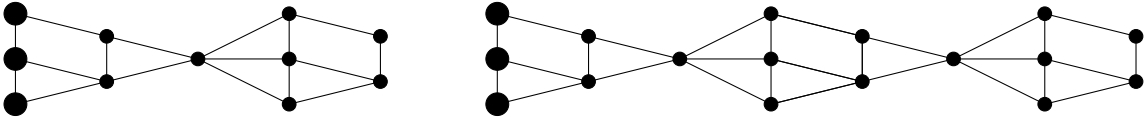
\begin{figure}[h]
    \centering

    \begin{tikzpicture}[scale=0.6]

    \foreach \x in {6}{
        \foreach \y in {-1,0,1}{
        \draw[fill] (\x,\y) circle (0.15);
        }
    }
    \foreach \x in {0}{
        \foreach \y in {-1,0,1}{
        \draw[fill] (\x,\y) circle (0.25);
        }
    }
    \foreach \x in {2,8}{
        \foreach \y in {-0.5,0.5}{
        \draw[fill] (\x,\y) circle (0.15);
        }
    }
    \foreach \x in {4}{
        \foreach \y in {0}{
        \draw[fill] (\x,\y) circle (0.15);
        }
    }
    
    \foreach \x in {0,6}{ 
        \draw (\x,1)--(\x,0);
        \draw (\x,-1)--(\x,0);
    }
    \foreach \x in {2}{ 
        \draw (\x,0.5)--(\x,-0.5);
        
        % \draw (\x,0.5)--(\x-2,0);
        \draw (\x,0.5)--(\x-2,1);
        \draw (\x,-0.5)--(\x-2,0);
        \draw (\x,-0.5)--(\x-2,-1);

        \draw (\x,-0.5)--(\x+2,0);
        \draw (\x,0.5)--(\x+2,0);
    }
    \foreach \x in {4}{ 
        \draw (\x,0)--(\x+2,-1);
        \draw (\x,0)--(\x+2,0);
        \draw (\x,0)--(\x+2,1);
    }
    \foreach \x in {8}{ 
        \draw (\x,0.5)--(\x,-0.5);
        
        % \draw (\x,0.5)--(\x-2,0);
        \draw (\x,0.5)--(\x-2,1);
        \draw (\x,-0.5)--(\x-2,0);
        \draw (\x,-0.5)--(\x-2,-1);
    }
    \end{tikzpicture} \quad\quad\quad
        \begin{tikzpicture}[scale=0.6]

    \foreach \x in {6,12}{
        \foreach \y in {-1,0,1}{
        \draw[fill] (\x,\y) circle (0.15);
        }
    }
    \foreach \x in {0}{
        \foreach \y in {-1,0,1}{
        \draw[fill] (\x,\y) circle (0.25);
        }
    }
    \foreach \x in {2,8,14}{
        \foreach \y in {-0.5,0.5}{
        \draw[fill] (\x,\y) circle (0.15);
        }
    }
    \foreach \x in {4,10}{
        \foreach \y in {0}{
        \draw[fill] (\x,\y) circle (0.15);
        }
    }
    
    \foreach \x in {0,6,12}{ 
        \draw (\x,1)--(\x,0);
        \draw (\x,-1)--(\x,0);
    }
    \foreach \x in {2,8}{ 
        \draw (\x,0.5)--(\x,-0.5);
        
        % \draw (\x,0.5)--(\x-2,0);
        \draw (\x,0.5)--(\x-2,1);
        \draw (\x,-0.5)--(\x-2,0);
        \draw (\x,-0.5)--(\x-2,-1);

        \draw (\x,-0.5)--(\x+2,0);
        \draw (\x,0.5)--(\x+2,0);
    }
    \foreach \x in {4,10}{ 
        \draw (\x,0)--(\x+2,-1);
        \draw (\x,0)--(\x+2,0);
        \draw (\x,0)--(\x+2,1);
    }
    \foreach \x in {8,14}{ 
        \draw (\x,0.5)--(\x,-0.5);
        
        % \draw (\x,0.5)--(\x-2,0);
        \draw (\x,0.5)--(\x-2,1);
        \draw (\x,-0.5)--(\x-2,0);
        \draw (\x,-0.5)--(\x-2,-1);
    }
    \end{tikzpicture}   
\caption{The graph $G$ (where vertices in $N_0$ are shown larger) and its concatenation $G'$}\label{fig:GAndConcatenation}
\end{figure}

In the current paper, we will focus on graphs with clique number $\omega \leq 3$ or chromatic number $\chi \leq 3$. For each integer $\delta \geq 4$, let $f(\delta)$ be the smallest rational number such that $\diam(G) \leq f(\delta)n+O(1)$ for all graphs $G$ with order $n$, minimum degree $\delta$ and clique number $\omega \leq 3$. Similarly, let $f'(\delta)$ be this number when the restriction $\omega \leq 3$ is replaced by $\chi \leq 3$ (so $f(\delta) \geq f'(\delta)$). 

To determine $f(\delta)$ (or $f'(\delta)$), it suffices to find the best repeatable graph. More precisely, even the best minimal repeatable graph would suffice (as defined in Subsec.~\ref{sec:notation}).

One direction is trivial: if there is a repeatable graph $G$ (with respect to $\delta,\omega$ or $\delta,\chi$) with repetition length $p$ and the graph induced by all layers of $G$ except the first and the last layer has order $n$, then we can concatenate $G$ several times (the resulting graph will have clique number at most $\omega$ or chromatic number at most $\chi$ and all vertices except for vertices in the first and last layer have degree at least $\delta$). Connecting every vertex from the first and last layer with the vertices of a (different) $K_{\delta, \delta}$ results in a graph with minimum degree at least $\delta$ and clique number bounded by three, $\omega \leq 3$ (or chromatic number at most three, $\chi \leq 3$). By repeatedly concatenating, we see that the asymptotic ratio between the diameter and the order goes to $\frac{p}{n}$, illustrating why it is a lower bound for $f(\delta)$ (or $f'(\delta)$).

Conversely, consider for every $d \ge 2$ the graph with diameter $d$ (under the $\delta$ and $\omega$ or $\chi$ condition) with minimum order.
For such a graph, $\abs{N_i} \le 2\delta,$ as otherwise one can replace $N_{i-1}$ and $N_{i+1}$ by an independent set of the same size and replace $N_i$ by a $K_{\delta, \delta}$ whose vertices are all connected to $N_{i-1}\cup N_{i+1}.$
This implies that there is a finite bound $B$ (depending only on $\delta$) for the number of possibilities $G[N_i \cup N_{i+1}]$. By the pigeonhole principle, this means that there exists a constant $C$ (also depending only on $\delta$) such that every graph (under the $\delta$ and $\omega$ or $\chi$ condition) with diameter at least $C$ must contain a subgraph induced by consecutive layers that is repeatable. 
Among all possible minimal repetition lengths, let $f$ be the best ratio of the repetition length to the order. One can remove repeatable graphs (if $G[ \cup_{i=r}^s N_i]$ is repeatable, removing it implies we replace $G$ by the graph $G'$ which is formed by adding edges to $G[\cup_{i=1}^{r} N_i] \cup G[\cup_{i=s}^{d} N_i]$ between $N_r$ and $N_s$ such that $G'[N_r \cup N_s] \cong G[N_r \cup N_{r+1}]$ and the isomorphism maps the vertices of $N_s$ to $N_{r+1}$)
one by one until the resulting graph's diameter is at most $C$.
As the order $n$ goes to infinity, one deduces from this that initially $\diam(G) \le fn+O(1)$ (if the repeatable graph is not minimal, then either it contains a shorter repeatable graph which has a ratio which is not worse, or removing that shorter repeatable graph results in a repeatable graph with a better ratio).

The rest of this paper is structured as follows. In~\cref{sec:omega3}, we consider the case where $\omega(G)\le 3$ and determine $f(4), f(5)$ and $f(6)$ exactly. The repetition length of this optimal repeatable graph becomes long, indicating that solving the question from~\cite{EPPT89} exactly in general is probably very difficult. We observe that the optimal repeatable graphs for $\delta=4,5,6$ have equal chromatic number and clique number, $\chi(G)=\omega(G)=3$, which we expect to hold in the general case. We therefore formulate the following conjecture.

\begin{conj}
    For every $\delta \ge 4$, $f(\delta)=f'(\delta)$.
\end{conj}

For some small orders, extremal graphs with $\chi(G)>\omega(G)$ exist, and they are not unique.

In~\cref{sec:chi3}, we consider the weaker variant where we concentrate on 3-colourable graphs and determine $f'(7)$ and $f'(8)$ exactly and a lower bound for $f'(16)$ (which is in fact also exact under some additional mild assumptions).

Czabarka, Singgih, and Sz\'ekely showed in~\cite{CSS23} that $\overline{f'(\delta)}=\frac{7}{3\delta}$ is an upper bound for $f'(\delta).$ We compare the exact values that we obtained with this general upper bound in~\cref{tab:delta/n_delta_chi3}.

\renewcommand{\arraystretch}{1.2}
\begin{table}[h]
    \centering
    \begin{tabular}{|c|c|c|c|}
        \hline
       $\delta$  & $f(\delta)$ & $f'(\delta)$ & $\overline{f'(\delta)}$\\ \hline
       4  & $\frac4{7}$ & $\frac{4}{7}$ & $\frac{7}{12}$\\
       5  & $\frac{5}{11}$ & $\frac{5}{11}$ & $\frac{7}{15}$\\
       6  & $\frac{14}{37}$ & $\frac{14}{37}$ & $\frac{7}{18}$\\
       7  & - & $\frac{17}{52}$ & $\frac{1}{3}$\\
       8  & - & $\frac{2}{7}$ & $\frac{7}{24}$\\
       % 9  &  & $\frac{7}{27}$ \\
       16 & -  & $\textcolor{red}{ \frac{31}{216} }$ & $\frac{7}{48}$\\
       \hline
    \end{tabular}
    \caption{A comparison between $f(\delta)$, $f'(\delta)$ and $\overline{f'(\delta)}$ for small values of $\delta$. The value in red is a lower bound for $f'(16)$, which is exact under additional mild assumptions.}
    \label{tab:delta/n_delta_chi3}
\end{table}

When $\omega$ or $\chi$ is at most 3, the bounds in~\cref{conj:EPPT_wrong} (i) are only stated when $8 \mid \delta$ and these were disproved in~\cite{CSS21} when $\delta \ge 24$. For $\delta \in \{8,16\},$ the conjectured bounds from~\cref{conj:EPPT_wrong} (i) yield $f(8) \leq \frac 27$ and $f(16) \leq \frac 17$. As such, the results in~\cref{tab:delta/n_delta_chi3} indicate that (likely) $f(8)=\frac27$ (as conjectured), but $ f(16) \geq f'(16)\geq \frac{31}{216} > \frac{1}{7}$, thereby yielding the first counterexample to~\cref{conj:EPPT_wrong} (i) in this regime. For the sake of clarity, we stress that we can show that $f'(16) = \frac{31}{216}$ under additional mild assumptions, but the inequality $f'(16) \geq \frac{31}{216}$ holds unconditionally. Finally, we remark that the asymptotic maximum diameter for a $k$-colourable graph is known to be of the form $\left( 3 - \Theta\right( \frac 1k \left) \right)\frac{n}{\delta}+O(1) $ by~\cite[Thm.~3]{CSS21} and~\cite[Thm.~5]{CSS21ejc}, but the exact determination of $f'_k(\delta)$ (and $f_k(\delta)$) -- the analogues for $f'(\delta)$ and $f(\delta)$ when $3$ is replaced by $k$ -- remains open.

\subsection{Notation and terminology}
\label{sec:notation}
We use mostly standard terminology, but explain additional notation in this subsection.

Let $G=(V,E)$ be a graph and $X, Y \subset V(G)$.
The graph $G[X]$ is the subgraph induced by the set $X$. This is the graph with vertex set $X$ and edge set $E \cap \binom{X}{2}.$\\
The graph $G[X,Y]$ is the bipartite graph with vertex set $X \cup Y$ and edge set $\{xy \in E(G)\colon x \in X, y \in Y\}.$
Equivalently, $G[X,Y]$ equals $G[X \cup Y]$ with the edges in $\binom{X}{2} \cup \binom{Y}{2}$ deleted.
It is complete iff $E(G[X,Y])=X \times Y.$
Note that $G[X,Y]$ is a (possibly strict) subgraph of $G[X\cup Y]$.

A repeatable graph $G\left[ \cup_{h=i}^j N_h\right]$ is minimal if there is no repeatable $G\left[  \cup_{h=i'}^{j'} N_h\right]$ where $i\le i'<j'\le j$ and $(i',j') \not=(i,j)$. For example, the two graphs in~\cref{fig:GAndConcatenation} are both repeatable, but only the leftmost graph is a minimal repeatable graph. For the (iterative) concatenation of a minimal repeatable graph $G\left[ \cup_{h=i}^j N_h\right]$, we will call the repetition length of $G$ the \textit{period} of the concatenation.\\
A part of a graph $G$ is \textit{periodic} with period $p$ if $G[ \cup_{j=i}^{i+p} N_i]$ only depends on $i \mod p$ for some $i_{\textrm{min}} \le i \le i_{\textrm{max}}-p$ and that part has no smaller $p$ satisfying the property.\footnote{So we use the term period for what others may call primitive period or fundamental period.} 
If $i_{\textrm{max}}-i_{\textrm{min}}>2,$ every $p+2$ consecutive layers form a repeatable graph.
Every $p$ consecutive layers form a fundamental block.
The ratio of diameter over order of the fundamental block contains the fraction (ratio) we seek.
In the paper, we typically consider the fundamental block as part of a larger periodic graph, particularly a repeatable graph, to which it belongs.

When studying $f'(\delta)$, we need to know how many colours and how many vertices of a certain colour class are present in a certain layer.
We denote by $c(i)$ the number of colours present in $G[N_i]$.

We can present (part of) a graph by means of a matrix, where each row represents a colour, and a column a different layer (corresponding with a neighbourhood). The corresponding maximal graphs are called clump graphs in~\cite{CSS21}. Note that here $c(i)$ equals the chromatic number of $G[N_i]$, since $G[N_i]$ is a complete multipartite graph.

A $\chi \times \ell$- matrix $A=(a_{i,j})_{i \in [\chi], j \in [\ell]}$ will represent a $\chi$-colourable graph of diameter $\ell-1$ (if $\ell\ge3$), which can be formed by independent sets $a_{i,j}K_1,$ where additionally the union of $a_{i,j}K_1$ and $a_{i,j'}K_1$ and the union of $a_{i,j}K_1$ and $a_{i+ 1,j'}K_1$ form complete bipartite graphs for every $i$ and $j \not=j'$, and no other additional edges are present.

$$A=\begin{pmatrix} 
 a_{1,1}& a_{1,2}& a_{1,3}& \ldots &a_{1,\ell}\\
 a_{2,1}& a_{2,2}& a_{2,3}&\ldots & a_{2,\ell}\\
\vdots &&&&\vdots\\
 a_{\chi, 1}& a_{\chi, 2}& a_{\chi, 3}&\ldots & a_{\chi,\ell}
\end{pmatrix}$$

\section{Maximum diameter for $K_4$-free graphs}\label{sec:omega3}

In this section, we determine $f(4), f(5)$ and $f(6)$ exactly. This is done in three subsections.

\subsection{Minimum degree $4$}

\begin{prop}
\label{prop:omega3Delta4}
    If $G$ is a $K_4$-free graph of order $n$ and has minimum degree $\delta \ge 4$, then $\diam(G) \le \frac 47 n + O(1).$ Furthermore this is sharp up to the determination of $O(1)$, which will depend on $ n \pmod 4$ for $n$ large.
\end{prop}

\begin{proof}
    Let $u,v \in V(G)$ be such that $\diam(G)=d(u,v)$ and let $N_i=N_i(u)$ for every $i.$
    \begin{claim}\label{clm:4/7}
        For every $0 \le i\le \diam(G)-3$, it must be that $\sum_{j=i}^{i+3} \abs{N_j} \ge 7.$
    \end{claim}
    \begin{claimproof}
    By the minimum degree condition $\sum_{j=i}^{i+2} \abs{N_j} \ge 5$ and $\sum_{j=i+1}^{i+3} \abs{N_j} \ge 5$.
    Hence if $\sum_{j=i}^{i+3} \abs{N_j} \le 6,$ 
    $\abs{N_i}=\abs{N_{i+3}}=1$ and $\abs{N_{i+1}}+\abs{N_{i+2}}=4.$
    But every vertex in $N_{i+1} \cup N_{i+2}$ can as such have at most $4$ neighbours, and equality implies that $N_{i+1} \cup N_{i+2}$ induces a clique $K_4.$
    So by $\delta \ge 4$ and $G$ being $K_4$-free, we conclude that $\sum_{j=i}^{i+3} \abs{N_j} \le 6$ is not the case.         
    \end{claimproof}
    By~\cref{clm:4/7}, we conclude that $n=\sum_{j=0}^{\diam(G)} \abs{N_j} \ge 7 \floorfrac{\diam(G)}4$ and thus $\diam(G) \le \ceilfrac{4n}{7}+3.$

    For sharpness, it is sufficient to consider a concatenation of a repeatable graph like in~\cref{fig:omega3delta4},
    where at the beginning and end one can append some complete bipartite graphs of correct size to adjust such that the order and minimum degree are correct.
\end{proof}

\begin{figure}[h]
    \centering

    \begin{tikzpicture}[scale=0.55]
    \foreach \x in {0,8}{
        \draw[fill] (\x,0) circle (0.15);
    
    }
    
    \foreach \x in {2,4,6,10}{
        \foreach \y in {1,-1}{
        \draw[fill] (\x,\y) circle (0.15);
    }
    }

    \foreach \x in {2}{
        
        \draw (\x+2,1)--(\x,1)--(\x,-1)--(\x+2,1); 
        \draw (\x+2,-1)--(\x,1)--(\x-2,0)--(\x,-1)--(\x+2,-1); 
    }

    \foreach \x in {6}{
        
        \draw (\x-2,1)--(\x,1)--(\x,-1)--(\x-2,1); 
        \draw (\x-2,-1)--(\x,1)--(\x+2,0)--(\x,-1)--(\x-2,-1); 
    }
    \foreach \x in {10}{
        
        \draw (\x,1)--(\x,-1); 
        \draw (\x,1)--(\x-2,0)--(\x,-1); 
    }
    
    \node at (11,0) {$\cdots$};
    \end{tikzpicture} 
\begin{tikzpicture}[scale=0.55]
    \foreach \x in {0,8,16}{
        \draw[fill] (\x,0) circle (0.15);
    
    }
    
    \foreach \x in {2,4,6,10,12,14}{
        \foreach \y in {1,-1}{
        \draw[fill] (\x,\y) circle (0.15);
    }
    }

    \foreach \x in {2,10}{
        
        \draw (\x+2,1)--(\x,1)--(\x,-1)--(\x+2,1); 
        \draw (\x+2,-1)--(\x,1)--(\x-2,0)--(\x,-1)--(\x+2,-1); 
    }

    \foreach \x in {6,14}{
        
        \draw (\x-2,1)--(\x,1)--(\x,-1)--(\x-2,1); 
        \draw (\x-2,-1)--(\x,1)--(\x+2,0)--(\x,-1)--(\x-2,-1); 
    }

    \end{tikzpicture}    
\caption{Repetitive concatenation of the repeatable graph (on the left)}\label{fig:omega3delta4}
\end{figure}
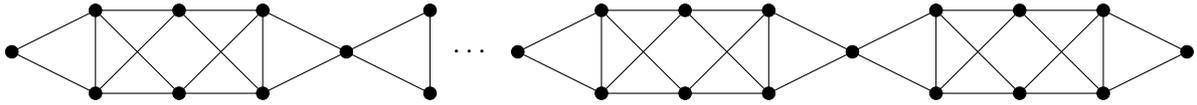
The fundamental block (repeated once in gray) can also be given as a clump graph with corresponding matrix

\begin{center}
    $\begin{pmatrix} 
1&0 & 2 & 0& \textcolor{lightgray}{1}&\textcolor{lightgray}0 & \textcolor{lightgray}2 &\textcolor{lightgray} 0\\
0&1 & 0&1&\textcolor{lightgray}{0}&\textcolor{lightgray}1 & \textcolor{lightgray}0&\textcolor{lightgray}1\\
0&1& 0 &1&\textcolor{lightgray}{0}& \textcolor{lightgray}1 & \textcolor{lightgray} 0&\textcolor{lightgray}1\\
\end{pmatrix}$
\end{center}

With some more work, one can verify the exact bounds, i.e., determine the $O(1)$ as well.
For $n \in \{6,7\}$ the diameter is two (so small orders behave slightly different).

\begin{prop}\label{prop:omega3Delta4exact}
     If $G$ is a $K_4$-free graph of order $n$ and has minimum degree $\delta \ge 4$, then $\diam(G) \le \floorfrac{4(n-4)}7 + 1_{\{n \in \{6,7,12\}\}}.$ 
     Furthermore, this bound is sharp.
\end{prop}

\begin{proof}
    It is easiest to prove this in the reverse direction.
    If $\diam(G)=d(u,v)=d\ge 3$ is of the form $4k+2,4k+3, 4k+4$ or $4k+5$, with $k\ge 1$, then the order of $G$ needs to be at least resp. $7k+8, 7k+10, 7k+11$ or $7k+13.$ This can be shown as follows. Note that both $\abs{N_0}+\abs{N_1}\ge 5$ and $\abs{N_{d-1}}+\abs{N_d} \ge 5.$
    Also any three consecutive neighbourhoods contain at least $5$ elements, and four have at least $7$. 
    By the above,
    \begin{itemize}
        \item if the graph has $4k+3$ layers, its order is at least $5+5+7(k-1)+5=7k+8,$\\
        (here we summed the lower bounds for the order of the first 2, next 3, following quadruples and final two neighbourhoods resp.)
        \item $4k+4$ layers result into an order at least $5+7k+5=7k+10,$
        \item if the graph has $4k+5$ layers, its order is at least $5+1+7k+5=7k+11,$
        \item $4k+6$ layers result into an order at least $5+5+5+7(k-1)+5=7k+13.$
    \end{itemize}
    
    Noting that the graph with diameter $d$ has $d+1$ layers, we conclude for $d \ge 6.$ The small values have been checked separately.

    Sharpness can be derived by inserting the gadget from~\cref{fig:omega3delta4} (repeatable graph minus one end layer which has size of neighbourhoods $(1,2,2,2,1)$) into the corresponding constructions from~\cref{fig:omega3delta4_exact} (at the single vertex of $N_{d-2}$), or enlarging $N_1.$ The small cases where $n \in \{6,7,10,12\}$ are easily verified as well (e.g. $K_6$ and $K_{5,5}$ (both) minus a perfect matching for $n=6$ and $n=10$).
\end{proof}

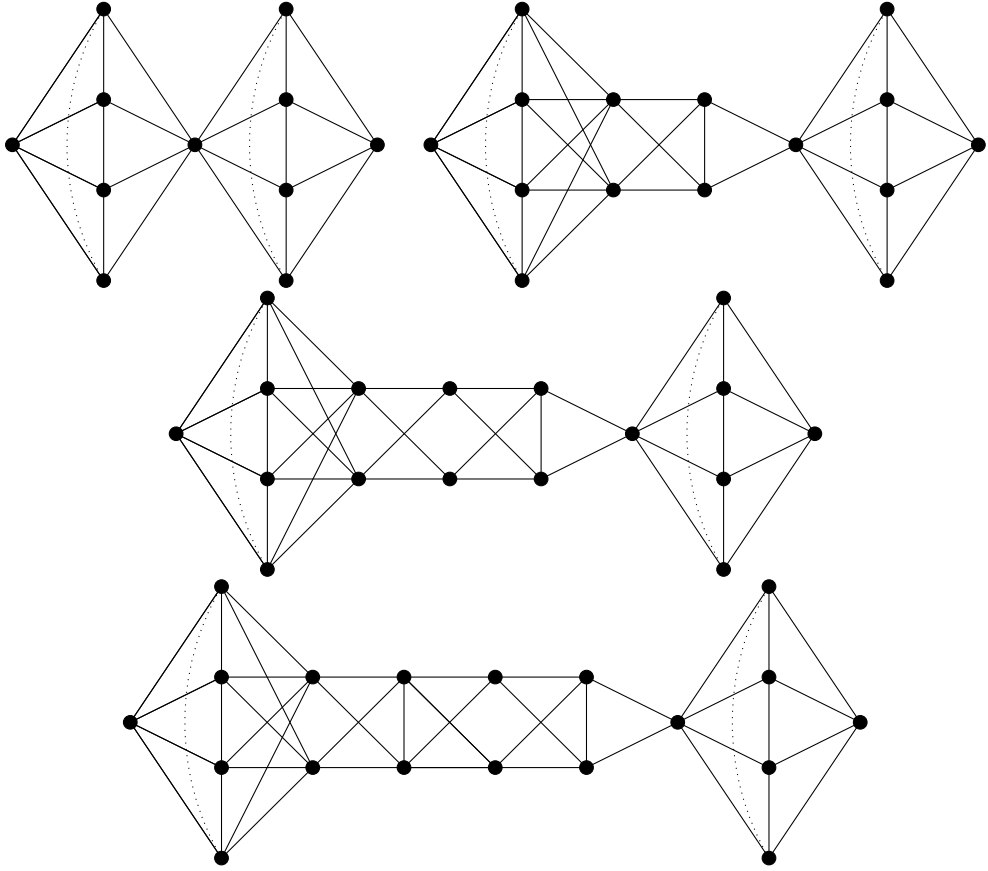
\begin{figure}[h]
    \centering

       \begin{tikzpicture}[scale=0.6]

    \foreach \x in {4,8,12}{
        \draw[fill] (\x,0) circle (0.15);
    
    }

    \foreach \x in {6,10}{
    \draw (\x,-3)--(\x,3);
     \draw[dotted] (\x,-3) arc (210:150:6) ;
        \foreach \y in {1,-1,-3,3}{
        \draw[fill] (\x,\y) circle (0.15);
    }
    }

    \foreach \y in {1,-1,-3,3}{
        \draw (4,0)--(6,\y);
    }

    \foreach \y in {1,-1,-3,3}{
        \draw(12,0)--(10,\y)--(8,0);
    }
    
    % \foreach \x in {4,6}{
    %     \foreach \y in {1,-1}{
    %     \draw[fill] (\x,\y) circle (0.15);
    % }
    % }

     \foreach \y in {1,-1,-3,3}{
        \draw(4,0)--(6,\y)--(8,0);
    }

    %     \draw (\x+2,1)--(\x,1)--(\x,-1)--(\x+2,1); 
    %     \draw (\x+2,-1)--(\x,1)--(\x-2,1)--(\x,-1)--(\x+2,-1); 
    %     \draw (\x+2,-1)--(\x,1)--(\x-2,-1)--(\x,-1)--(\x+2,-1);
    % }

    % \foreach \x in {6}{
        
    %     \draw (\x-2,1)--(\x,1)--(\x,-1)--(\x-2,1); 
    %     \draw (\x-2,-1)--(\x,1)--(\x+2,0)--(\x,-1)--(\x-2,-1); 
    % }

    \end{tikzpicture}
    \quad
    \begin{tikzpicture}[scale=0.6]

    \foreach \x in {0,8,12}{
        \draw[fill] (\x,0) circle (0.15);
    
    }

    \foreach \x in {2,10}{
    \draw (\x,-3)--(\x,3);
     \draw[dotted] (\x,-3) arc (210:150:6) ;
        \foreach \y in {1,-1,-3,3}{
        \draw[fill] (\x,\y) circle (0.15);
    }
    }

    \foreach \y in {1,-1,-3,3}{
        \draw (0,0)--(2,\y);
         \draw (0,0)--(2,\y);
    }

    \foreach \y in {1,-1,-3,3}{
        \draw(12,0)--(10,\y)--(8,0);
    }
    
    \foreach \x in {4,6}{
        \foreach \y in {1,-1}{
        \draw[fill] (\x,\y) circle (0.15);
    }
    }

    \foreach \x in {1,-1}{
     \foreach \y in {1,-1,-3,3}{
        \draw(4,\x)--(2,\y);
    }
    }
        
    %     \draw (\x+2,1)--(\x,1)--(\x,-1)--(\x+2,1); 
    %     \draw (\x+2,-1)--(\x,1)--(\x-2,1)--(\x,-1)--(\x+2,-1); 
    %     \draw (\x+2,-1)--(\x,1)--(\x-2,-1)--(\x,-1)--(\x+2,-1);
    % }

    \foreach \x in {6}{
        
        \draw (\x-2,1)--(\x,1)--(\x,-1)--(\x-2,1); 
        \draw (\x-2,-1)--(\x,1)--(\x+2,0)--(\x,-1)--(\x-2,-1); 
    }

    \end{tikzpicture} 

    \begin{tikzpicture}[scale=0.6]

    \foreach \x in {-2,8,12}{
        \draw[fill] (\x,0) circle (0.15);
    
    }

    \foreach \x in {0,10}{
    \draw (\x,-3)--(\x,3);
     \draw[dotted] (\x,-3) arc (210:150:6) ;
        \foreach \y in {1,-1,-3,3}{
        \draw[fill] (\x,\y) circle (0.15);
    }
    }

    \foreach \y in {1,-1,-3,3}{
        \draw (-2,0)--(0,\y);
         \draw (-2,0)--(0,\y);
    }

    \foreach \y in {1,-1,-3,3}{
        \draw(12,0)--(10,\y)--(8,0);
    }
    
    \foreach \x in {2,4,6}{
        \foreach \y in {1,-1}{
        \draw[fill] (\x,\y) circle (0.15);
    }
    }

    \foreach \x in {1,-1}{
     \foreach \y in {1,-1,-3,3}{
        \draw(2,\x)--(0,\y);
    }
    }
     \foreach \x in {4}{
        
        \draw (\x-2,1)--(\x,1)--(\x-2,-1); 
        \draw (\x-2,1)--(\x,-1)--(\x-2,-1); 
    }

    \foreach \x in {6}{
        
        \draw (\x-2,1)--(\x,1)--(\x,-1)--(\x-2,1); 
        \draw (\x-2,-1)--(\x,1)--(\x+2,0)--(\x,-1)--(\x-2,-1); 
    }

    \end{tikzpicture} 

    \begin{tikzpicture}[scale=0.6]

    \foreach \x in {-4,8,12}{
        \draw[fill] (\x,0) circle (0.15);
    
    }

    \foreach \x in {-2,10}{
    \draw (\x,-3)--(\x,3);
     \draw[dotted] (\x,-3) arc (210:150:6) ;
        \foreach \y in {1,-1,-3,3}{
        \draw[fill] (\x,\y) circle (0.15);
    }
    }

    \foreach \y in {1,-1,-3,3}{
        \draw (-4,0)--(-2,\y)--(0,1);
         \draw (-4,0)--(-2,\y)--(0,-1);
    }

    \foreach \y in {1,-1,-3,3}{
        \draw(12,0)--(10,\y)--(8,0);
    }
    
    \foreach \x in {0,2,4,6}{
        \foreach \y in {1,-1}{
        \draw[fill] (\x,\y) circle (0.15);
    }
    }

    \foreach \x in {2}{
        
        \draw (\x+2,1)--(\x,1)--(\x,-1)--(\x+2,1); 
        \draw (\x+2,-1)--(\x,1)--(\x-2,1)--(\x,-1)--(\x+2,-1); 
        \draw (\x+2,-1)--(\x,1)--(\x-2,-1)--(\x,-1)--(\x+2,-1);
    }

    \foreach \x in {6}{
        
        \draw (\x-2,1)--(\x,1)--(\x,-1)--(\x-2,1); 
        \draw (\x-2,-1)--(\x,1)--(\x+2,0)--(\x,-1)--(\x-2,-1); 
    }

    \end{tikzpicture} 

\caption{$K_4$-free graphs with $\delta=4$ and large diameter for small orders used in~\cref{prop:omega3Delta4exact}}\label{fig:omega3delta4_exact}
\end{figure}

\subsection{Minimum degree $5$}

\begin{prop}
    If $G$ is a $K_4$-free graph of order $n$ and has minimum degree $\delta \ge 5,$ then $\diam(G) \le \frac 5{11} n + O(1).$ 
\end{prop}

\begin{proof}
    We start determining an optimal period, and for this we first prove assumptions we may take into account for every neighbourhood $N_i$ within a fundamental block.
\begin{claim}\label{clm:4spansC4}
    If $N_i$ is of size 4, we can assume that
    \begin{itemize}
        \item $G[N_i]$ spans a $C_4$
        \item $N_{i-1}$ and $N_{i+1}$ form an independent set
        \item $G[N_i, N_{i+1}]$ and $G[N_i, N_{i-1}]$ are complete.
    \end{itemize}
\end{claim}

\begin{claimproof}
    If $\abs{N_{i-1}}=\abs{N_{i+1}}=1,$ we would have a $K_4$ in the center, which cannot happen.
    Hence $\abs{N_{i-1}}+\abs{N_{i+1}}\ge 3.$ 

    Now assume we replace $G[N_i]$ by a $C_4$, remove all edges in $G[N_{i-1}]$ and $G[N_{i+1}]$ to obtain independent sets and add an edge between every vertex $x \in N_i$ and every vertex $y \in N_{i-1} \cup N_{i+1}$.
    No $K_4$ has been created in this way.
    We can assume that every vertex in $N_{i+1}$ has at least one neighbour in $N_{i+2}$, since otherwise we just can add one such an edge, without creating a $K_4.$ 
    It is easily verified that every vertex in $N_{i-1}\cup N_i \cup N_{i+1}$ has degree at least $5.$
\end{claimproof}

\begin{claim}\label{clm:sum8_5ways}
    There are only $5$ possibilities where $\abs{N_i}+\abs{N_{i+1}}+\abs{N_{i+2}}+\abs{N_{i+3}}\le 8.$ 
    In those cases, $(\abs{N_i},\abs{N_{i+1}},\abs{N_{i+2}},\abs{N_{i+3}})$ is among $\{(1,1,4,2),(1,2,4,1),(1,3,3,1),(1,4,2,1),(2,4,1,1)\}$.
\end{claim}

\begin{claimproof}
    Note that $\abs{N_i}+\abs{N_{i+1}}+\abs{N_{i+2}}\ge \delta+1=6$, 
    $\abs{N_i}+\abs{N_{i+1}}+\abs{N_{i+2}}+\abs{N_{i+3}}\ge \delta+2=7$ and equality would imply that $\abs{N_i}=\abs{N_{i+3}}=1$ and $N_{i+1}\cup N_{i+2}$ is a clique, which is a contradiction.
    Hence the sum $\abs{N_i}+\abs{N_{i+1}}+\abs{N_{i+2}}+\abs{N_{i+3}}$ is at least $8$ and $\max \{\abs{N_i},\abs{N_{i+3}}\} \le 2.$
    It is easy to rule out $(1,1,5,1)$ or symmetrically $(1,5,1,1),$ as well as $(2,2,2,2)$.
    Also $(1,2,3,2)$, $(1,3,2,2)$ and $(1,4,1,2)$ (and the reflections) are impossible under the condition $\omega<4<\delta.$ By our assumption of~\cref{clm:4spansC4}, it is also clear that each of the $5$ cases with equality in~\cref{clm:sum8_5ways} corresponds with a unique part of a graph.
\end{claimproof}

Since $\frac 94>\frac {11}{5}$ (we later show that $\frac {5}{11} \le f(5)$), there exist $4$ consecutive neighbourhoods with sum of sizes equal to $8.$
In particular, we can consider a minimum optimal period and corresponding fundamental block.
If $\abs{N_i}=\abs{N_j}=1$ (where $i<j$) and
$\abs{N_{i+1}}=\abs{N_{j+1}}$, we know that $G[ \cup_{h=i}^{j+1} N_h]$ contains a fundamental block.

First, assume that no two consecutive size one neighbourhoods are present.
In that case, there are at most $3$ quadruples of consecutive neighbourhoods with sum of sizes $8$ in a fundamental block.
In that case the period is at least $6$ (size $1$ neighbourhoods need to be at distance at least $3$ of each other, and period $3$ gives the worse bound $\frac 73$) and thus at least twice the sum of $4$ consecutive neighbourhoods is $8$ and two such quadruples have to intersect.
The latter implies that without loss of generality we may assume that $\abs{N_{i}}=\abs{N_{i+3}}=\abs{N_{i+6}}=1.$
If $(\abs{N_{i+j}})_{0 \le j \le 6}=(1,2,4,1,4,2,1),$ there are two subsets of length $4$ with sum $11$ and we are done.
In the other case, $(1,3,3,1)$ borders (at least) one of the two other options and so at least one sum is $10.$
Hence the period is bounded by length $10$ (if not, the ratio is worse than $\frac{11}5$).
Hence the fundamental block is either of length $9$ and order at least $3 \cdot 7$,
or length $10$ and order at least $1+2\cdot7+9=24.$ Both result in worse ratios.

So assume $(\abs{N_{i}},\abs{N_{i+1}})=(1,1)$ occurs for some $i$ as part of the string $(4,1,1,4)$, which has sum $10.$
Since the fundamental block will have length at least $6,$ we need to have at least three times a sum of sizes of 4 consecutive neighbourhoods, which is $8.$
This again implies that the fundamental block has to be of length strictly larger than $6$.
Since the sum of $3$ consecutive neighbourhoods is at least $6$, each $4$ will be part of another sequence of $4$ consecutive neighbourhoods with sum at least $10.$
Taking into account~\cref{clm:sum8_5ways}, all $5$ ways with low sum need to be in the fundamental block, from which uniqueness of the optimal fundamental block follows.

The latter also gives sharpness.
It is sufficient to consider a (repeated) concatenation of the gadget (fundamental block) like in~\cref{fig:omega3delta5},
    where at the beginning and end one can append some bipartite graphs of correct size to adjust such that the order and minimum degree are correct.
\end{proof}

\begin{figure}[h]
    \centering

    \begin{tikzpicture}[scale=0.75]

  % draw all lines
        \foreach \y in {2,0,-2}{
        \draw (8,\y) --(6,0);
        \draw (10,\y) --(12,0);
    }
      \foreach \y in {1,-1}{
        \draw (4,\y) --(6,0);
        \draw (14,\y) --(12,0);
    }

 \foreach \x in {-1,1}{
    \foreach \y in {-1,1,-3,3}{
        
        \draw (4,\x)--(2,\y);
        \draw (14,\x)--(16,\y);
    }
    }

    \draw[dotted] (0,0)--(-1.5,0);
    \draw[dotted] (18,0)--(19.5,0);
    
    \foreach \y in {-1,1,-3,3}{
        
        \draw (0,0)--(2,\y);
        \draw (18,0)--(16,\y);
    }

    \foreach \x in  {2,0,-2}{
        \foreach \y in {2,0,-2}{
        
        \draw (8,\x)--(10,\y);
    }
    }

     \foreach \x in  {8,10}{
    \draw (\x,-2)--(\x,2);
    }

    \foreach \x in  {2,16}{
    \draw (\x,-3)--(\x,3);
    }

     \draw[red,thick] (8,0)--(10,0);

  %draw all vertices  
        \foreach \x in {0,6,12,18}{
        \draw[fill] (\x,0) circle (0.15);
        }    
    \foreach \y in {-1,1}{
            \foreach \x in {14,4}{
        \draw[fill] (\x,\y) circle (0.15);
        }
    }
      \foreach \y in {-1,1,-3,3}{
            \foreach \x in {16,2}{
        \draw[fill] (\x,\y) circle (0.15);
        }
    }
    \foreach \x in {8,10}{
        \foreach \y in {2,0,-2}{
        \draw[fill] (\x,\y) circle (0.15);
    }
    }

    \draw[dotted] (2,-3) arc (210:150:6) ;
 \draw[dotted] (16,-3) arc (-30:30:6) ;

    \end{tikzpicture} 

\caption{The optimal graph for $\delta=5$.  The red line draws attention to the missing edge.}\label{fig:omega3delta5}
\end{figure}
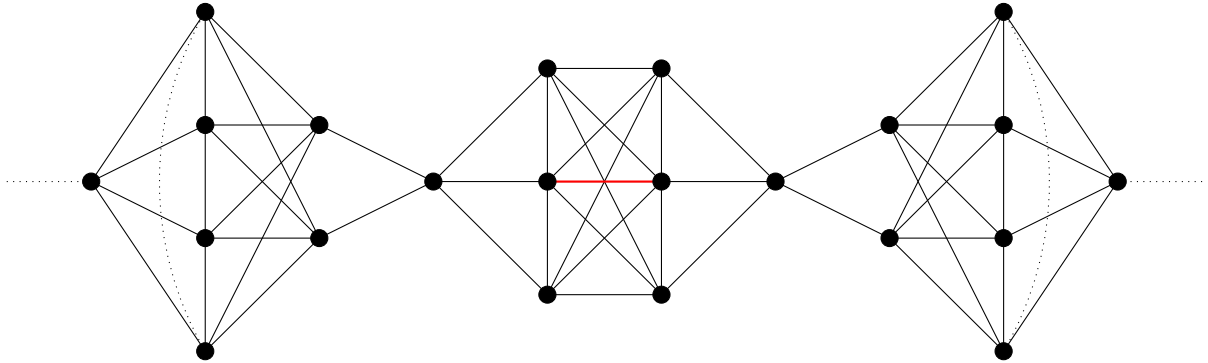

The fundamental block itself is a clump graph that can be presented by the corresponding matrix (from the gray column onwards, the repetition starts)

$$\begin{pmatrix} 
1&0 & 2 & 0&1&1&0&0&2&0& \textcolor{lightgray}{1}\\
0&2 & 0&0&2&0&1&0&2&0&\textcolor{lightgray}{0}\\
0&2& 0 &1&0&2&0&2&0&1&\textcolor{lightgray}{0}\\
\end{pmatrix}$$

\subsection{Minimum degree $6$}

\begin{prop}
    If $G$ is a $K_4$-free graph of order $n$ and has minimum degree $\delta \ge 6,$ then $\diam(G) \le \frac {14}{37} n + O(1).$ 
\end{prop}

\begin{proof}

\begin{claim}\label{clm:5spansC5}
    If $N_i$ is of size 5, we can assume that
    \begin{itemize}
        \item $G[N_i]$ spans a $K_{2,3}$ (or $C_5$)
        \item $N_{i-1}$ and $N_{i+1}$ form an independent set
        \item $G[N_i, N_{i+1}]$ and $G[N_i, N_{i-1}]$ are complete.
    \end{itemize}
\end{claim}

\begin{claimproof}
    If $\abs{N_{i-1}}=\abs{N_{i+1}}=1,$ we would have a $K_5$ in the center, which cannot happen.
    Also in the case $\abs{N_{i-1}}+\abs{N_{i+1}}= 3,$ with some case distinction one concludes that $\delta \ge 6$ implies that there is a $K_4$ (each vertex in $N_i$ has at most one non-neighbour among the other $7$ vertices).
    
    Hence $\abs{N_{i-1}}+\abs{N_{i+1}}\ge 4 .$ 

    Now assume we replace $G[N_i]$ by a $C_5$ or $K_{2,3}$, remove all edges in $G[N_{i-1}]$ and $G[N_{i+1}]$ to obtain independent sets and add an edge between every vertex $x \in N_i$ and every vertex $y \in N_{i-1}$. No $K_4$ has been created in this way. We can assume that every vertex in $N_{i+1}$ has at least one neighbour in $N_{i+2}$, since otherwise we just can add one such an edge, without creating a $K_4.$ 
    It is easily verified that every vertex in $N_{i-1}\cup N_i \cup N_{i+1}$ has degree at least $6.$
\end{claimproof}

\begin{claim}\label{clm:less_than6}
     We can assume that there is no neighbourhood with $\abs{N_i}>5$ in the optimal fundamental block. Consequently, we also assume there is no $i$ with $\abs{N_{i-1}}+\abs{N_{i+1}}\le 3$.
     % except possibly at the start and end.
\end{claim}

\begin{claimproof}
    If $\abs{N_{i-1}}=\abs{N_{i+1}}=1$, then $\abs{N_i}\ge 8$ is needed since $G[N_i]$ has to be triangle-free.
    We can replace their sizes by $2,5,2$ (taking into account~\cref{clm:5spansC5}).
    If $\abs{N_{i-1}}+\abs{N_{i+1}}\ge 3$ and $\abs{N_i}\ge 6$, we can replace 
    $G[N_{i\pm 1}]$ by an independent set of size $\max\{ \abs{N_{i\pm 1}},2\}$, replace $G[N_i]$ by $K_{3,2}$, and let $G[N_i, N_{i+1}]$ and $G[N_i, N_{i-1}]$ be complete.
\end{claimproof}

\begin{claim}\label{clm:44}
     If in the optimal fundamental block, $\abs{N_i}=\abs{N_{i+1}}=4$, we can assume that $\{ G[N_i], G[N_{i+1}]\} \in \{ \{S_4,S_4\}, \{4K_1, 4K_1\}, \{4K_1, C_4\} \}. $
\end{claim}

\begin{claimproof}
    If $\abs{N_{i-1}}=\abs{N_{i+2}}=1$, we can let $G[N_i \cup N_{i+1}]$ be a balanced tripartite graph $T(8,3)$ (all edges between two copies of a 4-vertex-star $S_4$ are present, except for the one between the two centers) and let $G[N_{i+2}, N_{i+1}]$ and $G[N_i, N_{i-1}]$ be complete.

    If $\abs{N_{i-1}}=1$ and $\abs{N_{i+2}}\ge 2$ (the reverse is analogous) we can choose $ G[N_i] = C_4, G[N_{i+1}]=4K_1$, where $ G[N_{i}, N_{i+1}]$ and $G[N_i, N_{i-1}]$ are complete, and after possibly removing an edge from $G[N_{i+2}]$ to make $G[N_{i+2}]$ triangle-free if necessary (there are only a few cases to consider by~\cref{clm:5spansC5} and~\cref{clm:less_than6}), $G[N_{i+2}, N_{i+1}]$ is also complete.

    If $\abs{N_{i-1}}, \abs{N_{i+2}}\ge 2,$ we can choose $\{ G[N_i], G[N_{i+1}]\} = \{4K_1, 4K_1\}$, $ G[N_{i}, N_{i+1}]=K_{4,4}$, and both $G[N_{i+2}, N_{i+1}]$ and $G[N_i, N_{i-1}]$ being complete bipartite as well (after possibly removing an edge from $G[N_{i-1}]$ and/or $G[N_{i+2}]$ to make it triangle-free).
\end{claimproof}

Finally, using the algorithm described in~\cref{sec:pc_search}, we can find the optimal period thanks to the aforementioned claims, yielding $f(6)=\frac{14}{37}$.

An example of an optimal fundamental block is presented below

$$\begin{pmatrix} 
0& 1& 0& 3& 0& 1& 0& 3& 0& 2& 0& 2& 0& 3\\
2& 0& 0& 2& 0& 0& 2& 0& 1& 0& 2& 0& 1& 0\\
2& 0& 2& 0& 2& 0& 2& 0& 1& 0& 2& 0& 1& 0
\end{pmatrix}$$

One can note that $N_i \cup N_{i+1}$ spans a complete bi- or tripartite graph, but it is not necessarily a Turán graph since there is e.g. an appearance of $K_{3,1,1}.$
\end{proof}

\section{Maximum diameter for $3$-colourable graphs}\label{sec:chi3}

In this section, we prove that the statement obtained when restricting ~\cref{conj:EPPT_wrong} (i) to graphs with chromatic number at most $\chi$ is correct when $r=2$ if and only if $\delta=8$.

More precisely, we first prove that 

\begin{prop}\label{prop:delta7}
    If $G$ is a $3$-colorable graph of order $n$ with minimum degree $\delta \ge 7,$ then $\diam(G) \le \frac {17}{52} n + O(1).$ 
\end{prop}

\begin{prop}\label{prop:delta8}
    If $G$ is a $3$-colorable graph of order $n$ with minimum degree $\delta \ge 8,$ then $\diam(G) \le \frac {2}{7} n + O(1).$ 
\end{prop}

We start with an easy observation (proving that having all colours present in a neighbourhood implies a local relative deficit), which is also true when $\omega=3$ and $N_i$ contains a triangle.

\begin{claim}\label{clm:localdeficit_withtriangle}
    If $c(i)=3,$ then $\abs{N_{i-1}}+\abs{N_i}+\abs{N_{i+1}} \ge \ceilfrac{3\delta}{2}.$
\end{claim}
\begin{claimproof}
    The sum of sizes of two colour classes among $N_{i-1}, N_i$ and $N_{i+1}$ is at least $\delta$.
Summing over the three combinations, leaves us with $2(\abs{N_{i-1}}+\abs{N_i}+\abs{N_{i+1}}) \ge 3\delta.$
\end{claimproof}

Analogous statements of~\cref{clm:localdeficit_withtriangle} hold for larger $\chi$ or $\omega$ as well.

We give a more precise upper bound than the earlier mentioned crude $2\delta$ bound on the order of the neighbourhoods within an optimal fundamental block.

\begin{claim}\label{clm:|Ni|<delta}
    For $\delta \in \{7,8\}$, we can assume that there is no neighbourhood for which $\abs{N_i}\ge \delta$ in the optimal fundamental block.
\end{claim}

\begin{claimproof}
    This can be verified by case analysis. The details are explained in~\cref{sec:cases}, which extends the proof for $\delta=8$ we sketch here.

    Assume the claim is not true, and thus $\abs{N_i} \ge \delta$. 
    First observe that $\abs{N_i} \le \floorfrac{3 \delta}2$, since a balanced neighbourhood ($3$ colour classes with sizes that differ at most $1$) of size $\floorfrac{3 \delta}2$ can be fitted in anywhere.

    If $c(i-2), c(i), c(i+2)\le 2$, then one can assume that $c(i-1)=c(i+1)=1$.
    If $\abs{N_{i-1}}+\abs{N_{i+1}} \ge \delta,$ it is trivial that the size of $N_{i}$ can be decreased without destroying the property.
    In the other case, one can move vertices from $N_i$ to $N_{i\pm 1}$ till $\abs{N_i}=\delta -1$ and end with a fundamental block that is still fine.

    If $c(i-2), c(i+2)\le 2$ and $c(i)=3$, we can again put all vertices from $N_{i \pm 1}$ in a single colour class.
    By~\cref{clm:localdeficit_withtriangle}, $\abs{N_{i-1}}+\abs{N_i}+\abs{N_{i+1}} \ge \ceilfrac{3\delta}{2}.$
    Now one can take a balanced two-colouring of $N_i$ where $\abs{N_i}=\delta-1,$ possibly after moving some vertices to $N_{i \pm 1}.$
    
    If $c(i-2)=c(i+2)=3$, one can remove $N_{j}$ for $j \in [i-3,i+3]\setminus\{i\}$ and put a balanced $3$-coloured $N_i$ of size $10$ between $N_{i-4}$ and $N_{i+4}.$ 
    By~\cref{clm:localdeficit_withtriangle}, we removed at least $2\cdot 12+ 8$ vertices, replacing them by $10$, giving a decrease of at least $22$ vertices, while the diameter decreases by only $6$.
    Since $\frac{22}{6}>\frac 72,$ this is an improvement.

    Finally assume $c(i-2)=3$ and $c(i+2)\le 2.$ Using~\cref{clm:localdeficit_withtriangle} and considering a few cases, we can decrease the order by at least $8$ and the length by $2$.
\end{claimproof}

Knowing restrictions on the sizes of all colour classes, we can once again use the algorithm described in~\cref{sec:pc_search} to obtain $f'(7)=\frac{17}{52}$ and $f'(8)=\frac{2}{7}$. This concludes the proof for~\cref{prop:delta7} and~\cref{prop:delta8}.

An optimal fundamental block for respectively $\delta=7$ and $\delta=8$ is given below. Here every $t \in \{1,2,3\}$ works for $\delta=8.$

$$\begin{pmatrix} 
0& 3& 0& 1& 0& 3& 0& 3& 0& 1& 0& 3& 0& 1& 0& 3& 0\\
0& 3& 0& 0& 3& 0& 1& 0& 3& 0& 2& 0& 2& 0& 3& 0& 1\\
3& 0& 1& 0& 3& 0& 0& 1& 2& 0& 0& 3& 0& 0& 2& 1& 0
 \end{pmatrix} \quad \begin{pmatrix} 
t&0& 6-t& 0\\
0& 2& 0& 2\\
0& 2& 0& 2
\end{pmatrix} $$

The following assumption seems natural for the extremal graphs, but has not been proven.

\textbf{Assumption}
In the optimal fundamental block for $\chi=k \ge 3$ for some $\delta$, we may assume that $c(i) \le k-1$ for every $i$.

For $\delta=16$, the optimal period cannot be computed in reasonable time using the algorithm described in~\cref{sec:pc_search} if this assumption is not made. However, by using the previous assumption and additionally assuming that the period is bounded by 100 and that there is an $i$ for which $c(i)=1$, we obtain a counterexample~for $\delta=16$ to~\cref{conj:EPPT_wrong} (i) (in a regime outside of the regime of the counterexamples produced by Czabarka, Singgih, and Sz\'ekely~\cite{CSS21}).
From this, we may also expect that the region for which the conjecture holds is narrower than the narrowed window by the authors from~\cite{CSS21} (end of page 39). The fundamental block that yields the counterexample is presented below, resulting in the fraction $\frac{31}{216}$.
\[
\resizebox{\textwidth}{!}{$
\begin{pmatrix} 
0 & 0 & 7 & 0 & 1 & 0 & 7 & 0 & 0 & 7 & 0 & 0 & 2 & 5 & 0 & 0 & 7 & 0 & 0 & 4 & 3 & 0 & 0 & 7 & 0 & 0 & 8 & 0 & 1 & 0 & 7 \\
0 & 1 & 0 & 8 & 0 & 7 & 0 & 2 & 0 & 7 & 0 & 2 & 0 & 7 & 0 & 4 & 0 & 5 & 0 & 7 & 0 & 2 & 0 & 7 & 0 & 2 & 0 & 7 & 0 & 0 & 8 \\
2 & 0 & 7 & 0 & 1 & 0 & 7 & 0 & 2 & 0 & 7 & 0 & 7 & 0 & 2 & 0 & 7 & 0 & 2 & 0 & 7 & 0 & 6 & 0 & 3 & 0 & 5 & 2 & 0 & 7 & 0
\end{pmatrix}
$}
\]

For $9 \le \delta \le 15$, one could in principle obtain sharp bounds under similar assumptions, by modifying the computer code in~\cref{sec:pc_search}, but $\delta=16$ was the next interesting case after $\delta=8$ with respect to~\cref{conj:EPPT_wrong} (i) and so we decided not to pursue determining those values.

\section*{Acknowledgements}
\noindent The computational resources and services used in this work were provided by the VSC (Flemish Supercomputer Centre), funded by the Research Foundation Flanders (FWO) and the Flemish Government - Department EWI.

\section*{Appendix}\label{sec: appendix}

\appendix

\section{Details of some case analysis for~\cref{clm:|Ni|<delta}}\label{sec:cases}

\begin{claim}\label{clm:delta-1underassumption}
If in the optimal fundamental block $c(i-2), c(i+2)\le 2$, then $\abs{N_i}\le \delta-1$.
\end{claim}

\begin{claimproof}
    This can be verified by case analysis.

    First observe that $\abs{N_i} \le \floorfrac{3 \delta}2$, since a balanced neighbourhood of that size can fit anywhere. So assume $\abs{N_i}\ge \delta$.
    We consider two cases.

    \begin{itemize}
    \item If $c(i)\le 2$, then one can assume that $c(i-1)=c(i+1)=1$.
    Note that we can permute the colours of $N_i$ and $N_j$ for $j \ge 2$ such that the same colour is missing in $N_{i-2}, N_i$ and $N_{i+2}.$
    So we can put $\abs{N_{i\pm 1}}$ many vertices in the third colour at $N_{i\pm 1}$.
    Every vertex in $N_j$ for $\abs{j-i}\not= 1$ has degree at least $\delta$ as this was initially the case.
    Every vertex in $N_{i\pm 1}$ also has degree at least $\delta$ since $\abs{N_i}+\abs{N_{i\pm 2}}  \ge \delta.$

    If $\abs{N_{i-1}}+\abs{N_{i+1}} \ge \delta,$ it is trivial that the size of $N_{i}$ can be decreased (one can even choose $c(i)=1$ and $\abs{N_i}=\delta -1$) while all vertices keep having degree at least $\delta.$
    If $\abs{N_{i-1}}+\abs{N_{i+1}} <\delta,$ one can move vertices from $N_i$ to $N_{i\pm 1}$ till $\abs{N_i}=\delta -1$ and end with a construction for which the minimum degree is still at least $\delta$.

    \item If $c(i)=3$, we can again put all vertices from $N_{i \pm 1}$ in a single colour class.
    By~\cref{clm:localdeficit_withtriangle}, $\abs{N_{i-1}}+\abs{N_i}+\abs{N_{i+1}} \ge \ceilfrac{3\delta}{2}.$
    Now one can take a balanced two-colouring of $N_i$ where $\abs{N_i}=\delta-1,$ possibly after moving some vertices to $N_{i \pm 1}.$ \qedhere
    \end{itemize}
    \end{claimproof}

    \begin{claim}
        If $\delta \in \{7,8\},$ then every neighbourhood in an optimal fundamental block satisfies $\abs{N_i}\le \delta-1.$
    \end{claim}

    \begin{claimproof}
        By~\cref{clm:delta-1underassumption}, we need to focus on two cases, which we do for $\delta=7$ and $\delta=8$ separately.

        \textbf{Case $\delta=7$}

        \begin{itemize}
            \item If $c(i-2)=c(i+2)=3$, one can remove $N_{j}$ for $j \in [i-3,i+3]\setminus\{i\}$ and put a balanced $3$-coloured $N_i$ of size $9$ between $N_{i-4}$ and $N_{i+4}.$ 
    By~\cref{clm:localdeficit_withtriangle}, we removed at least $2\cdot 11+ 7$ vertices, replacing them by $9$, giving a decrease of at least $20$ vertices, while the diameter decreases by only $6$.
    Since $\frac{20}{6}>3,$ this is an improvement.

    \item We assume $c(i-2)=3$ and $c(i+2)\le 2$ (the reverse is analogous).
    As before, we can assume $c(i+1)=1.$
    
    Using~\cref{clm:localdeficit_withtriangle}, we know that $\abs{N_{i-3}}+\abs{N_{i-2}}+\abs{N_{i-1}}\ge 11.$
    We also have $\abs{N_i} \ge 7.$
    
    Depending on $c(i-4)$ being $1,2$ or $3$, we can perform different substitutions that imply an improvement of ratio given by the period divided by $3$.
    If $c(i-4)=3,$ we can put consecutively $N_{i-4}, N_i, N_{i+1}$ to be at least $[1,1,1],[3,3,1],[0,0,1].$
    If $c(i-4)=2,$ we can do the same with $[0,1,1],[3,3,2],[0,0,1].$

    If $c(i-4)=1,$ we can put $[0,0,1],[3,3,0],[0,0,3],[3,3,0],[0,0,1]$ for $N_{i-4}$ up to $N_{i+1}$, decreasing the order with at least $3$ while the diameter decreases by $1$ and the number of neighbourhoods with $\abs{N_i}\ge \delta$ decreased by at least one.

    This is presented in~\cref{fig:localupdates7}. 
    Here we present the matrix $A$, where every column represents the number of vertices in each colour class for a neighbourhood.

    \begin{figure}[h]
 \centering
 $\begin{pmatrix} 
0 & 1 & \textbf{3} & 0&0\\
1 & 1&\textbf{3}&0&1 \\
1 & 1 & \textbf{1} & 1&0
\end{pmatrix}$
\quad
 $\begin{pmatrix} 
1 & 0 & \textbf{3} & 0&0\\
0 & 1&\textbf{3}&0&1 \\
0 & 1 & \textbf{2} & 1&0
\end{pmatrix}$
\quad
 $\begin{pmatrix} 
 0 &\textbf{3} & 0&\textbf{3}&0\\
 0&\textbf{3}&0&\textbf{3}&0 \\
 1 & 0 & \textbf{3}&0&1
\end{pmatrix}$
\caption{Examples of local improvements where $\delta=7$ and the diameter decreases}
 \label{fig:localupdates7}
\end{figure}
        \end{itemize}

          \textbf{Case $\delta=8$}

        \begin{itemize}
            \item If $c(i-2)=c(i+2)=3$, one can remove $N_{j}$ for $j \in [i-3,i+3]\setminus\{i\}$ and put a balanced $3$-coloured $N_i$ of size $10$ between $N_{i-4}$ and $N_{i+4}.$ 
    By~\cref{clm:localdeficit_withtriangle}, we removed at least $2\cdot 12+ 8$ vertices, replacing them by $10$, giving a decrease of at least $22$ vertices, while the diameter decreases by only $6$.
    Since $\frac{22}{6}>\frac 72,$ this is an improvement.

    \item Finally assume $c(i-2)=3$ and $c(i+2)\le 2.$ Using~\cref{clm:localdeficit_withtriangle}, we have $\abs{N_{i-3}}+\abs{N_{i-2}}+\abs{N_{i-1}}+\abs{N_i}\ge 20.$

    For $c(i-4)\in \{3,2,1\}$ resp., we can make local modifications, replacing $N_{i-3..i}$ with a single neighbourhood. These are presented in~\cref{fig:localupdates8}.
    Up to permuting, the $0$s and $1$s (or non-bold $2$) are lower bounds for the corresponding number of vertices.
    Here we use~\cite[Thm.~7(iii)]{CSS21ejc}), which says that if $c(i)=3$, then $c(i \pm 1) \ge 2$.

 %\begin{center}
\begin{figure}[h]
 \centering
 $\begin{pmatrix} 
0 & 1 & \textbf{4} & 0&0\\
1 & 1&\textbf{3}&0&1 \\
1 & 1 & \textbf{2} & 1&0
\end{pmatrix}$
\quad
 $\begin{pmatrix} 
1 & 0 & \textbf{4} & 0&0\\
0 & 1&\textbf{3}&0&1 \\
0 & 1 & \textbf{2} & 1&0
\end{pmatrix}$
\quad
$\begin{pmatrix} 
0 & 0 & \textbf{4} & 0&0\\
0 & 1&\textbf{3}&0&1 \\
1 & 0 & \textbf{3} & 1&0
\end{pmatrix}$
\quad
$\begin{pmatrix} 
0 & 0 & \textbf{4} & 0&0\\
1 & 0&\textbf{3}&0&1 \\
0 &  2 & \textbf{2} &  1&0
\end{pmatrix}$
% \quad
%  $\begin{pmatrix} 
%  0 &\textbf{3} & 0&\textbf{4}&0&1\\
%  0&\textbf{3}&0&0&\textbf{3}&0 \\
%  1 & 0 & \textbf{4}&0&\textbf{3}&0
% \end{pmatrix}$
\caption{Examples of modifications where the diameter decreases by $3$ when $\delta=8$}
 \label{fig:localupdates8}
\end{figure}
%\end{center}

Finally, we prove that the latter is impossible.
If $\abs{N_{i-5}}\le 2$, we need that at least $6$ vertices of $N_{i-3}$ are coloured by $2$ colours.
But since $c(i-2)\ge 3$ and $c(i-1)\ge 2$,
not every colour can appear $4$ times in $N_{i-3..i-1}.$

If $\abs{N_{i-5}}\ge 3$, we can end by a final modification, which results in a decrease of the order of $1$ and results in $\abs{N_i} \le 7.$

If the diameter decreases by $3$ and the number of vertices by at least $11,$ we know that the construction was not optimal.
So the only remaining case is when $c(i-4)=1$, $\abs{N_{i-4}}=\abs{N_{i+1}}=1$, $\abs{N_{i-3}}+\abs{N_{i-2}}+\abs{N_{i-1}}=12$ and $\abs{N_i}=8.$
We will show that this is impossible.

Let $x_j, y_j, z_j$ be the number of vertices in $N_j$ coloured with the first, second and third colour respectively.
Without loss of generality, we have $z_{i-4}=1$ and consequently $x_{i-3}, y_{i-3} \ge 1$ and $z_{i-3}=0$ (since $c(i-3)\ge 2$ and the assumptions that the colours of adjacent neighbourhoods are as disjoint as possible). 
One can represent this with the following part of the matrix

 \begin{center}
   $\begin{pmatrix} 
 0 &x_{i-3}&x_{i-2}&x_{i-1}& x_i \\
 0&y_{i-3}&y_{i-2}&y_{i-1}&y_i \\
 1&0&z_{i-2}&z_{i-1} &z_i 
\end{pmatrix}$ 
 \end{center}

Due to the minimum degree condition for the vertices in $N_{i-3}$ coloured by the first colour, we have that $y_{i-3}+y_{i-2}+z_{i-2}\ge 7$ and analogously 
$x_{i-3}+x_{i-2}+z_{i-2}\ge 7$.
Due to the minimum degree condition for the vertices in $N_{i-2},$ we have that $x_{i-3}+x_{i-2}+x_{i-1}= y_{i-3}+y_{i-2}+y_{i-1}=z_{i-2}+z_{i-1}=4.$
From combining these, $y_{i-3}+y_{i-2}\ge 7-4=3$ and
$x_{i-1}+y_{i-1}+z_{i-1}\le 12-7-3=2.$
But the latter implies that every vertex in $N_{i}$ has at least $8-2-1=5$ neighbours within $N_{i}$, while $c(i) \le 2$ ($c(i)=3$ leads to a contradiction with $c(i+1)=1$) and $\abs{N_i}=8$, as desired. \qedhere
\end{itemize}
\end{claimproof}

\section{Details about computer search}\label{sec:pc_search}

Given integers $\delta$ and $C$, we describe an algorithm that can be used to determine $f(\delta)$, assuming that $f(\delta)$ is determined by a repeatable graph $G$ (with respect to $\delta$ and $\omega=3$) whose repetition length is at most $C$. More precisely, $f(\delta)$ is then equal to the ratio of the repetition length of $G$ and the order of the graph induced by all layers of $G$ except for the first and last layer. Later, we then explain how $f'(\delta)$ can be computed by slightly modifying this algorithm.

The idea is that the algorithm builds repeatable graphs by adding layers $N_i$ of a graph one by one. Recall from before that we may assume without loss of generality that the number of vertices in each layer $N_i$ of a repeatable graph has an upper bound (for example $2\delta$ is such a valid upper bound that works in general, but better upper bounds are possible). The algorithm maintains the invariant that each vertex, except for vertices in the first and last layer, must have degree at least $\delta$ and the entire graph must have clique number $\omega \leq 3$. When adding edges between layers $N_{i-1}$ and $N_i$, it suffices to only consider adding edge sets $E \subseteq N_{i-1} \times N_i$ such that it is impossible to add another edge $e \in N_{i-1} \times N_i$ (where $e \notin E$) without resulting in a graph with clique number $\omega>3$, because more edges lead to larger vertex degrees and have no further influence on the clique number when adding additional layers. We will refer to such an edge set $E$ as a \textit{maximal edge set}. Moreover, the algorithm does not need to consider all combinations of layers $N_0, N_1,...,N_i$ and maximal edge sets between them. More precisely, given a graph $G$, which is induced by the consecutive layers $N_0, N_1, \ldots, N_i$. When adding further layers to $G$ in order to arrive at an optimal repeatable graph, the only parameters which are relevant consist of what the graphs $G[N_0], G[N_1], G[N_0 \cup N_1]$ and $G[N_i]$ are, together with the information of the number of layers, which degree each vertex in $N_i$ has (in the graph $G[N_{i-1} \cup N_i]$) and how many vertices $G$ has (less is better with respect to optimal repeatable graphs). This naturally lends itself to a dynamic programming approach, where one calculates the minimum order of a graph induced by consecutive layers for each combination of feasible parameters from the previous sentence. In case a repeatable graph is found, the algorithm updates the best ratio between repetition length and order of the graph induced by all layers except the first and the last one, and finally the algorithm returns the optimal such ratio $f(\delta)$. The pseudo code of the algorithm can be found in~\cref{algo:calculate_cDelta} (the main function) and~\cref{algo:recursivelyAddLayers} (the function that recursively adds layers).

The value $f'(\delta)$ can in fact be computed using an algorithm very similar to the original one. However, this version can be significantly sped up. More precisely, instead of considering which graph is induced by layer $N_i$, it suffices to know how many vertices of each colour class are present for the $\chi$-version. Given the number of vertices in each colour class in each layer, the edges are also automatically determined: we add an edge between each vertex $v \in N_i$ and each vertex $u \in N_{i-1} \cup N_{i} \cup N_{i+1}$ such that $u$ and $v$ belong to different colour classes (this does not affect the chromatic number, while making the degrees as large as possible). In other words, we only need to consider clump graphs (as defined in Subsec.~\ref{sec:notation}). This makes it possible to calculate $f'(\delta)$ for larger values than $f(\delta)$ can be computed.

Finally, we stress that the algorithms are also adapted to incorporate the Claims made in the main part of the paper (all algorithmic ideas remain the same, but the graphs that one needs to consider in each layer can be further restricted thanks to these claims). The algorithms were also parallelised to make the computations feasible. The total time of all computations performed in this paper amounts to approximately 1 CPU-year. We make all code publicly available at \url{https://github.com/JorikJooken/diameterDegreeClique}.

\begin{algorithm}[ht!]
\caption{\texttt{Calculate\_}$f(\delta)$(Integer $\delta$, Integer $C$)}
\label{algo:calculate_cDelta}
  \begin{algorithmic}[1]
            \STATE Let $N$ be an upper bound for the number of vertices in each layer $N_i$
            \STATE Let $\mathcal{L}$ be a list of pairwise non-isomorphic graphs with order at most $N$
            \STATE $f(\delta) \gets -\infty$
            \FOR{$G_1 \in \mathcal{L}$}
                \FOR{$G_2 \in \mathcal{L}$}
                    \FOR{Each maximal edge set $E$ between $G_1$ and $G_2$}
                        \STATE Let $G'$ be the graph obtained by adding each edge in $E$ to the disjoint union of $G_1$ and $G_2$
                        \IF{$G'$ has clique number $\omega \leq 3$}
                            \STATE $\texttt{parameters.}G[N_0] \gets G_1$
                            \STATE $\texttt{parameters.}G[N_1] \gets G_2$
                            \STATE $\texttt{parameters.}G[N_0 \cup N_1] \gets G'$
                            \STATE $\texttt{parameters.}G[N_{\texttt{lastLayer}}] \gets G_2$
                            \STATE $\texttt{parameters.numberLayers} \gets 2$
                            \STATE $\texttt{parameters.degreesLastLayer} \gets \{(u,\deg_{G'}(u))~|~u \in V(G_2)\}$
                            \STATE $\texttt{currentOrder} \gets |V(G')|$
                            \STATE $\texttt{bestRatio} \gets -\infty$ // A global variable that can be updated by the function $\texttt{recursivelyAddLayers}$
                            \STATE $\texttt{recursivelyAddLayers}(\texttt{parameters},\delta,C,\texttt{currentOrder})$

                            \STATE $f(\delta) \gets \max(f(\delta),\texttt{bestRatio})$
                        \ENDIF
                    \ENDFOR
                \ENDFOR
            \ENDFOR
            \RETURN $f(\delta)$
  \end{algorithmic}
\end{algorithm}

\begin{algorithm}[ht!]
\caption{\texttt{recursivelyAddLayers}(Parameters $p$, Integer $\delta$, Integer $C$, Integer \texttt{currentOrder})}
\label{algo:recursivelyAddLayers}
  \begin{algorithmic}[1]
            \IF{$p.\texttt{numberLayers}-1 \leq C$}
                \IF{The dynamic programming table $T$ does not contain any graph with the same parameters as $p$ and fewer vertices as \texttt{currentOrder}}
                    \STATE // Add one layer such that the new last layer is given by $G_{\texttt{last}}$
                    \FOR{$G_{\texttt{last}} \in \mathcal{L}$}
                        \FOR{Each maximal edge set $E$ between $p\texttt{.}G[N_{\texttt{lastLayer}}]$ and $G_{\texttt{last}}$}
                            \STATE Let $G'$ be the graph obtained by adding each edge in $E$ to the disjoint union of $p\texttt{.}G[N_{\texttt{lastLayer}}]$ and $G_{\texttt{last}}$
                            \IF{The degree of every vertex in $p\texttt{.}G[N_{\texttt{lastLayer}}]$ is at least $\delta$ after adding the edges from $E$ AND $G'$ has clique number $\omega \leq 3$}
                                \STATE $\texttt{newOrder} \gets \texttt{currentOrder}+|V(G_{\texttt{last}})|$
                                \STATE $\texttt{newP} \gets \texttt{updateParameters}(p,G_{\texttt{last}},E)$
                                \STATE $\texttt{updateDynamicProgrammingTable}(T,\texttt{newOrder},\texttt{newP})$
                                \STATE // Update \texttt{bestRatio}
                                \IF{The new graph is repeatable}
                                    \STATE $\texttt{bestRatio} \gets \max\left(\texttt{bestRatio},\frac{\texttt{newP}.\texttt{numberLayers}-2}{\texttt{newOrder}-|V(\texttt{newP}.G[N_0])|-|V(\texttt{newP}.G[N_{\texttt{lastLayer}}])|}\right)$
                                \ENDIF
                                \STATE $\texttt{recursivelyAddLayers}(\texttt{newP},\delta,C,\texttt{newOrder})$
                            \ENDIF
                        \ENDFOR
                    \ENDFOR
                \ENDIF
            \ENDIF
  \end{algorithmic}
\end{algorithm}

% \section{Some plausible sharp bounds}\label{sec:cases}

\end{document}